\documentclass[10pt]{amsart}
\usepackage{amssymb}
\usepackage{amsmath,amssymb,amsfonts,amsthm,graphics,
latexsym, amscd, amsfonts, epsfig, eepic,epic}
\usepackage{mathrsfs}
\usepackage{color}
\usepackage{eucal}
\usepackage{eufrak}
\usepackage[all]{xypic}
\usepackage{xspace}

\theoremstyle{plain}
\newtheorem{theorem}{Theorem}

\theoremstyle{definition}

\theoremstyle{example}

\theoremstyle{remark}
\newtheorem{remark}{Remark}
\numberwithin{equation}{section}

\begin{document}

\title[On $k$-noncrossing partitions]
      {On $k$-noncrossing partitions}
\author{Emma Y. Jin, Jing Qin and Christian M. Reidys$^{\,\star}$}
\address{Center for Combinatorics, LPMC-TJKLC \\
         Nankai University  \\
         Tianjin 300071\\
         P.R.~China\\
         Phone: *86-22-2350-6800\\
         Fax:   *86-22-2350-9272}
\email{reidys@nankai.edu.cn}
\thanks{}
\keywords{partition, $k$-noncrossing, $2$-regular, enhanced partition, braid,
         tangled-diagram,
          difference equation}
\date{October, 2007}
\begin{abstract}
In this paper we prove a duality between $k$-noncrossing partitions over
$[n]=\{1,\dots,n\}$ and $k$-noncrossing braids over $[n-1]$. This duality
is derived
directly via (generalized) vacillating tableaux which are in correspondence
to tangled-diagrams \cite{Reidys:07vac}.
We give a combinatorial interpretation of the bijection in terms of the
contraction of arcs of tangled-diagrams.
Furthermore it induces by restriction a bijection between $k$-noncrossing,
$2$-regular partitions over $[n]$ and $k$-noncrossing braids without
isolated points over $[n-1]$.
Since braids without isolated points correspond to enhanced partitions
this allows, using the results of \cite{MIRXIN}, to enumerate $2$-regular,
$3$-noncrossing partitions.
\end{abstract}
\maketitle {{\small
}}


\section{Introduction and Background}\label{S:1}


In this paper we prove a duality between $k$-noncrossing partitions and
braids, a particular type of tangled-diagrams \cite{Reidys:07vac}.
The duality implies a bijection between $2$-regular, $k$-noncrossing
partitions
and $k$-noncrossing braids without isolated points, which are in bijection
to enhanced partitions. We then compute the number of $3$-noncrossing,
$2$-regular partitions over $[n]=\{1,\dots,n\}$, i.e.~$k$-noncrossing
partitions without arcs of the form $(i,i+1)$. The enumeration
of $3$-noncrossing, $2$-regular partitions is not entirely trivial. This
is due to the fact that the lack of $1$-arcs translates into an asymmetry
induced by the nonexistence of the pair of steps ($(\varnothing,+\square_1),
(-\square_1,\varnothing)$), where ``$\,\pm\square_i\,$'' denotes the
adding/removing of a square in the $i$th row of the shape.
We derive the above duality directly via the (generalized) vacillating
tableaux \cite{Reidys:07vac} and prove its combinatorial interpretation
in terms of the contraction of arcs, originally introduced by
Chen~{\it et.al.}~in \cite{Chen-reduction} in the context of a reduction
algorithm for noncrossing partitions.

Our results imply novel connections between different combinatorial objects
and are of conceptual interest.
For instance, Bousquet-M\'{e}lou and Xin \cite{MIRXIN} have enumerated
$3$-noncrossing partitions and $3$-noncrossing enhanced partitions
separately, using kernel methods in nontrivial calculations.
By construction enhanced partitions correspond to hesitating tableaux
\cite{Chen} which accordingly enumerate braids without isolated points.
Our duality theorem implies therefore that either one of these
computations would imply the other.
Furthermore our results integrate the concepts of vacillating and hesitating
tableaux due to Chen~{\it et.al.}~\cite{Chen}.
$2$-regular partitions are of particular importance in the context of
enumerating RNA tertiary structures with base triples \cite{Reidys:073d}.

\section{Tangled-diagrams and vacillating tableaux}
\label{S:2}

In this Section we provide some basics on tangled-diagrams \cite{Reidys:07vac}.
A tangled-diagram is a labeled  graph, $G_n$, over $[n]$ with degree
$\le 2$, represented by drawing its vertices in a horizontal line and
its arcs $(i,j)$ in the upper halfplane having the following properties:
two arcs $(i_1,j_1)$ and $(i_2,j_2)$ such that $i_1<i_2<j_1<j_2$ are crossing
and if $i_1<i_2<j_2<j_1$ they are nesting.
Two arcs $(i,j_1)$ and $(i,j_2)$ (common lefthand endpoint) and
$j_1<j_2$ can be drawn in two ways: either draw $(i,j_1)$ strictly below
$(i,j_2)$ in which case $(i,j_1)$ and $(i,j_2)$ are nesting (at $i$) or
draw $(i,j_1)$ starting above $i$ and intersecting $(i,j_2)$ once,
in which case $(i,j_1)$ and $(i,j_2)$ are crossing (at $i$):
\begin{center}
\scalebox{0.25}[0.25]{\includegraphics*[30,710][560,810]{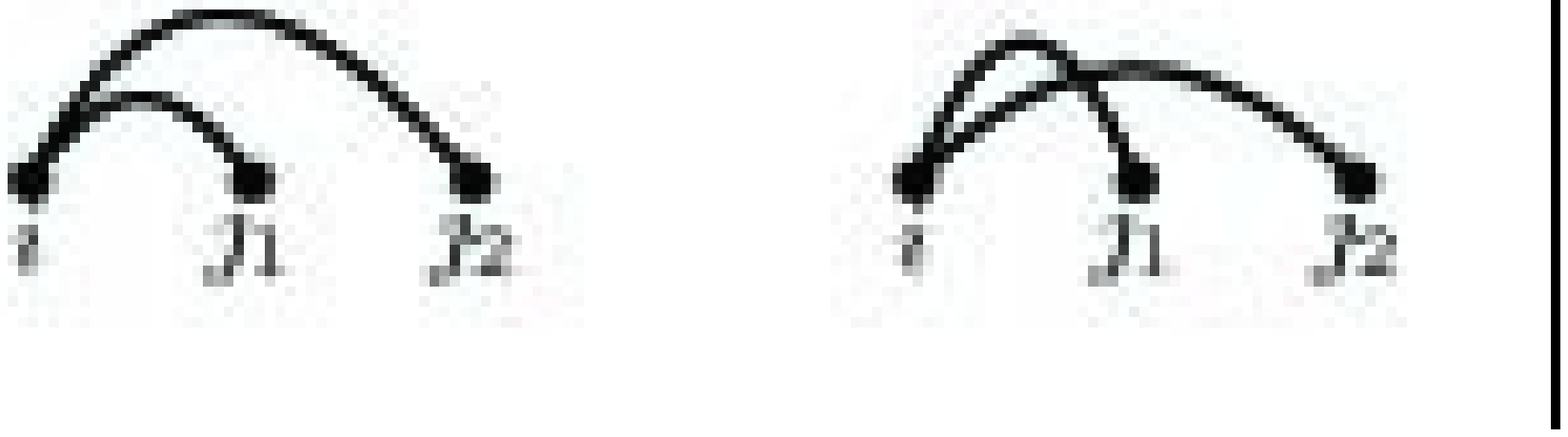}}
\end{center}
and of two arcs $(i,j),(i,j)$, i.e.~where $i$ and $j$ are both:
right- and lefthand endpoints are completely analogous. Suppose
$i<j<h$ and that we are given two arcs $(i,j)$ and $(j,h)$. Then we
can draw them intersecting once or not. In the former case $(i,j)$
and $(j,h)$ are called crossing, in the latter noncrossing arcs:
\begin{center}
\scalebox{0.6}[0.6]{\includegraphics*[60,780][560,830]{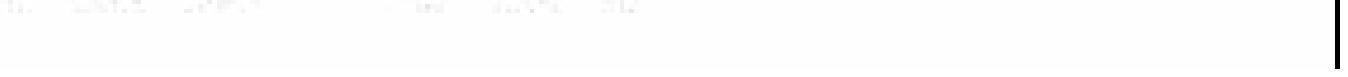}}
\end{center}
A $k$-noncrossing braid is a $k$-noncrossing tangled-diagram in
which all vertices $j$ of degree two are either incident to loops
$(j,j)$ or crossing arcs $(i,j)$ and $(j,h)$, where $i<j<h$. We
denote the set of $k$-noncrossing braids over $[n]$ by
$\mathcal{B}_k(n)$. For instance
\begin{center}
\scalebox{0.6}[0.6]{\includegraphics*[60,780][560,830]{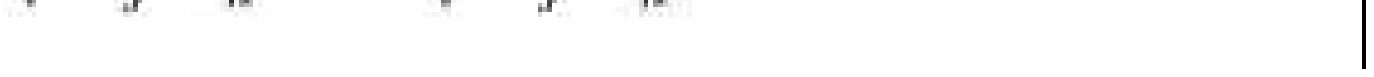}}
\end{center}
A shape is a collection of squares, ``$\,\square\,$'', arranged in
left-justified rows with weakly decreasing number of squares in each
row.
A vacillating tableaux $V_{\lambda}^{2n}$ of shape $\lambda$ and
length $2n$ is a sequence $(\lambda^{0}, \lambda^{1},\ldots,
\lambda^{2n})$ of shapes such that
{\sf (i)} $\lambda^{0}=\varnothing$ and $\lambda^{2n}=\lambda,$ and
{\sf (ii)} $(\lambda^{2i-1},\lambda^{2i})$ is derived from
           $\lambda^{2i-2}$, for $1\le i\le n$ by either
$(\varnothing,\varnothing)$: do nothing twice;
$(-\square,\varnothing)$: first remove a square then do nothing;
$(\varnothing,+\square)$: first do nothing then add a square;
$(\pm \square,\pm \square)$: add/remove a square at the odd and even steps,
respectively. Let $\mathcal{V}_\lambda^{2n}$ denote the set of vacillating
tableaux, for instance,
\begin{center}
\scalebox{0.6}[0.6]{\includegraphics*[10,750][580,830]{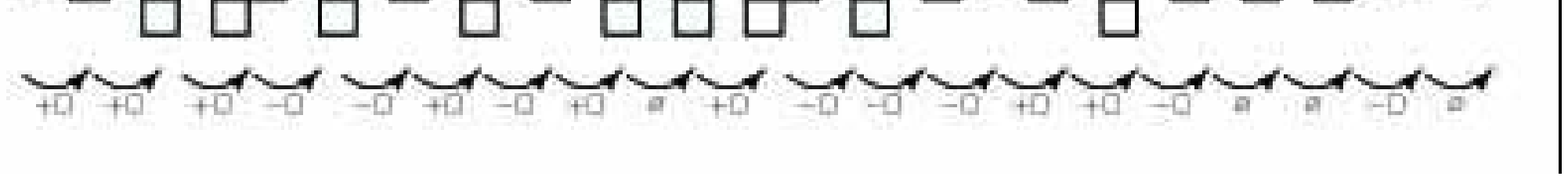}}
\end{center}
We have the following bijection between tangled-diagrams and
generalized vacillating tableaux \cite{Reidys:07vac} which
integrates the notions of vacillating and hesitating tableaux of
Chen~{\it et.al.}~\cite{Chen}. In the following we refer to
generalized vacillating tableaux simply as vacillating tableaux.
\begin{theorem}\label{T:bij}
There exists a bijection between the set of vacillating tableaux of shape
$\varnothing$ and length $2n$, $\mathcal{V}_\varnothing^{2n}$ and the
set of tangled-diagrams over $n$ vertices, $\mathcal{G}_n$
\begin{equation}
\beta\colon \mathcal{V}_{\varnothing}^{2n}  \longrightarrow
\mathcal{G}_n \ .
\end{equation}
Furthermore, a tangled-diagram $G_n$ is $k$-noncrossing if and only if all
shapes $\lambda^i$
in its corresponding vacillating tableau have less than $k$ rows,
i.e.~$\beta\colon \mathcal{V}_\varnothing^{2n}\longrightarrow \mathcal{G}_n$
maps vacillating tableaux having less than $k$ rows into $k$-noncrossing
tangled-diagrams. Furthermore there exists a bijection between the set of
$k$-noncrossing and $k$-nonnesting tangled-diagrams.
\end{theorem}
Restricting the set of generating step-pairs of vacillating tableaux recovers
the bijections of Chen {\it et.al.}~\cite{Chen}:
Let $M=\{(\varnothing,\varnothing),(-\square,\varnothing),
(\varnothing,+\square)\}$, $\mathcal{V}_{P,k,\varnothing}^{2n}$ and
$\mathcal{V}_{B,k,\varnothing}^{2n}$ denote the set of tableaux
with less than $k$ rows and generated by
$P=M\dot\cup \{(-\square,+\square)\}$ and
$B  = M \dot\cup \{(+\square,-\square)\}$.
Theorem~\ref{T:bij} allows us to identify $\mathcal{V}_{P,k,\varnothing}^{2n}$
with $\mathcal{P}_k(n)$ and $\mathcal{V}_{B,k,\varnothing}^{2n}$ with
$\mathcal{B}_k(n)$.
For partitions and braids we have the following correspondences between the
elementary pair-steps and associated tangled-diagram arc-configurations:
\begin{center}
\scalebox{0.5}[0.5]{\includegraphics*[30,760][560,820]{yyy15.eps}}
\end{center}
\begin{center}
\scalebox{0.5}[0.5]{\includegraphics*[40,760][560,820]{yyy16.eps}}
\end{center}


\section{Main results}\label{S:3}

We now prove the duality between partitions over $[n]$ and braids over
$[n-1]$. {\it A posteriori} the above bijection can be proved directly.
However, we arrived at this interpretation studying vacillating tableaux of
$k$-noncrossing partitions and braids.

\begin{theorem}\label{T:k-noncross}
Let $k\in\mathbb{N}$, $k\ge 3$. Then we have the bijection
\begin{equation}\label{E:biject1}
\vartheta\colon \mathcal{P}_{k}(n)\longrightarrow   \mathcal{B}_k(n-1) \ ,
\end{equation}
where $\vartheta$ has the following property:
for any $\pi\in\mathcal{P}_k(n)$ holds:
$(i,j)$ is an arc of $\pi$ if and only if $(i,j-1)$ is an arc in
$\vartheta(\pi)$.
\end{theorem}
\begin{proof}
A $k$-noncrossing partition $\pi$ corresponds via Theorem~\ref{T:bij}
uniquely to a vacillating tableaux, $V_{\varnothing}^{2n}(\pi)=
(\lambda^{i})_{i=0}^{2n}$. Let $\pm \square_h$ denote the adding or
subtracting of the rightmost square ``$\,\square\,$'' in the $h$th
row in a given shape $\lambda$ and let ``$\,\varnothing\,$'' denote
doing nothing.
$(\lambda^{i})_{i=0}^{2n}$ uniquely corresponds to a sequence of pairs
$\sigma_\pi=((x_i,y_i))_{i=1}^{n}$ where $(x_i,y_i)\in\{(\varnothing,
\varnothing), (-\square_j,+\square_h),(\varnothing,+\square_h),
(-\square_h,\varnothing)\}$, $1\le h,j\le k-1$ and
$x_1=y_{n}=\varnothing$. In the following we shall identify the
sequence $(x_i,y_i)_{i=1}^{n}$ with its corresponding sequence of shapes
and set
\begin{equation}\label{E:varphi}
\varphi_1((x_i,y_i)_{i=1}^{n})=(\tilde{x}_{i},\tilde{y}_i)_{i=1}^{n-1} \
\quad \text{\rm where}\quad \tilde{x}_i = y_{i}\ \wedge \
\tilde{y}_i=x_{i+1} \ .
\end{equation}
In view of $x_1=y_{n}=\varnothing$ we can conclude that $\varphi_1$ is
bijective. Since the vacillating tableaux of a partition is
generated by $(-\square,\varnothing)$, $(\varnothing,+\square)$,
$(\varnothing,\varnothing)$, $(-\square,+\square)$, we have
\begin{equation}\label{E:neu}
\forall\, 1\le i\le n-1;\qquad
(\tilde{x}_i,\tilde{y}_i)\in
\{(\varnothing,\varnothing),\  (+\square_h,\varnothing), \
(\varnothing,-\square_h), \ (+\square_h,-\square_j)\} \ ,
\end{equation}
where $1\le h,j\le k-1$. Let $\varphi_2$ be given by
\begin{equation}\label{E:varphi2}
\varphi_2((\tilde{x}_i,\tilde{y}_i))=
\begin{cases}
(\tilde{x}_i,\tilde{y}_i) & \ \text{\rm for } \
(\tilde{x}_i,\tilde{y}_i)=(+\square_h,-\square_j) \\
(\tilde{y}_i,\tilde{x}_i) & \ \text{\rm otherwise.}
\end{cases}
\end{equation}
$\varphi_2$ has by definition the property
$\varphi_2((\tilde{x}_i,\tilde{y}_i))\in \{(-\square_h,\varnothing),
(\varnothing,+\square_h),(\varnothing,\varnothing),(+\square_h,-\square_j)\}$.
\\
{\it Claim $1$.} The mapping
$$
\vartheta\colon \mathcal{P}_{k}(n)\longrightarrow   \mathcal{B}_k(n-1) \ ,\quad
\vartheta=\beta\circ \varphi_2\circ\varphi_1
\circ\beta^{-1}
$$
is well-defined and a bijection.\\
For arbitrary $\pi\in \mathcal{P}_k(n)$ we set
$V_\varnothing^{2n}(\pi)^\dagger=\varphi_2\circ \varphi_1(
V_\varnothing^{2n}(\pi))$.
By construction $V_\varnothing^{2n}(\pi)^\dagger$ is given by
$
\varphi_2(\varphi_1(x_i,y_i)_{i=1}^n))=(a_i,b_i)_{i=1}^{n-1}
$, where
$(a_i,b_i)\in \{(-\square_h,\varnothing),(\varnothing,+\square_h),
(\varnothing,\varnothing),(+\square_h,-\square_j)\}$.
Its induced sequence of collections of rows of squares
$(\mu^i)_{i=0}^{2(n-1)}$ has the following properties:
\begin{eqnarray}\label{E:w}
&& \mu^{2(n-1)}=\lambda^{2n-1}=\varnothing,\\
\label{E:right}
&& \mu^{2j+2}\setminus\mu^{2j+1},\; \mu^{2j+1}\setminus \mu^{2j}\in
\{(\varnothing,\varnothing), (\varnothing,+\square), (-\square,\varnothing),
(+\square,-\square)\} \\
\label{E:shape}
&& \mu^{2j+1}\neq\lambda^{2j+2}\quad \Longrightarrow\quad
\mu^{2j+1}\in \{\lambda^{2j+1},\lambda^{2j+3}\} \ .
\end{eqnarray}
Eq.~(\ref{E:w}) is obvious and eq.~(\ref{E:right}) follows from
eq.~(\ref{E:neu}). By construction of $(\mu^i)_{i=0}^{2(n-1)}$,
for $1\le j\le n-1$, $\mu^{2j} = \lambda^{2j+1}$ holds.
Suppose $\mu^{2j+1}\neq\lambda^{2j+2}$ for
some $0\le j\le n-2$. By definition of $\varphi_2$ only pairs containing
``$\varnothing$'' in at least one coordinate are transposed from which we
can conclude $\mu^{2j+1}=\mu^{2j}$ or $\mu^{2j+1}=\mu^{2j+2}$, i.e.
$$
\diagram
& \lambda^{2j+1} \ar@{=}[dl]\rto^{} & \lambda^{2j+2}\ar@{-}[dl]
\rto^{} &
\lambda^{2j+3} \ar@{=}[dl]  \\
 \mu^{2j} \ar@{=}[r] & \mu^{2j+1} \rto^{} &\mu^{2j+2}
\enddiagram
 \text{\rm or }
\diagram
& \lambda^{2j+1} \ar@{=}[dl]\rto^{} & \lambda^{2j+2}\ar@{-}[dl]
\rto^{} &
\lambda^{2j+3} \ar@{=}[dl]  \\
 \mu^{2j} \rto^{} & \mu^{2j+1} \ar@{=}[r] &\mu^{2j+2}
\enddiagram
$$
whence eq.~(\ref{E:shape}). In particular each collection of rows of
squares $\mu^i$ is a shape, i.e.~$V_\varnothing^{2n}(\pi)^\dagger$
corresponds to a braid.
Eq.~(\ref{E:shape}) immediately implies that $(\mu^{i})_{1\le
i\le 2(n-1)}$ has at most $k-1$ rows if and only if $(\lambda^{i})_{1
\le i\le 2n}$ does. Therefore $\vartheta$ is well-defined. Obviously
$\vartheta$ is bijective and Claim $1$ follows.\\
{\it Claim $2$.} For any $\pi\in\mathcal{P}_k(n)$ holds: $(i,j)$ is an arc
of $\pi$ if and only if $(i,j-1)$ is an arc in $\vartheta(\pi)$.\\
From the proof of Theorem~\ref{T:bij} \cite{Reidys:07vac} we know that
a $\pi$- and $\vartheta(\pi)$-origin at $j$ is equivalent to the
existence of a ``$\,+\square\,$'' in the pair-step between the shapes
$\lambda^{2j-1}$ and $\lambda^{2j}$ and $\mu^{2j-2}$ and $\mu^{2j}$,
respectively. We have the following alternative
$$
\diagram
& \lambda^{2j-1} \ar@{=}[dl]\rto^{+\square} & \lambda^{2j}\ar@{-}[dl]
\rto^{-\square} &
\lambda^{2j+1} \ar@{=}[dl]  \\
 \mu^{2j-2} \rto^{+\square} & \mu^{2j-1} \rto^{-\square} &\mu^{2j}
\enddiagram
\quad
\diagram
& \lambda^{2j-1} \ar@{=}[dl]\rto^{+\square} &
\lambda^{2j}\ar@{-}[dl] \rto^{\varnothing} &
\lambda^{2j+1} \ar@{=}[dl]  \\
 \mu^{2j-2} \rto^{\varnothing} & \mu^{2j-1} \rto^{+\square} &\mu^{2j}
\enddiagram
$$
It is clear by diagram-chasing that $\pi$ has an origin at $j$ if and
only if $\vartheta(\pi)$ does.
The situation changes however for endpoints of arcs. A $\pi$- and
$\vartheta(\pi)$-endpoint at $j$ is equivalent to a ``$\,-\square\,$'' in the
pair-step between $\lambda^{2j-2}$ to $\lambda^{2j-1}$ and $\mu^{2j-2}$
to $\mu^{2j}$, respectively. Therefore we have the following two
situations:
$$
\diagram
& \lambda^{2j-1} \ar@{=}[dl]\rto^{+\square} & \lambda^{2j}\ar@{-}[dl]
\rto^{-\square} &
\lambda^{2j+1} \ar@{=}[dl]  \\
 \mu^{2j-2} \rto^{+\square} & \mu^{2j-1} \rto^{-\square} &\mu^{2j}
\enddiagram
\quad
\diagram
& \lambda^{2j-1} \ar@{=}[dl]\rto^{\varnothing} &
\lambda^{2j}\ar@{-}[dl] \rto^{-\square} &
\lambda^{2j+1} \ar@{=}[dl]  \\
 \mu^{2j-2} \rto^{-\square} & \mu^{2j-1} \rto^{\varnothing} &\mu^{2j}
\enddiagram
$$
Again by diagram-chasing we immediately conclude that $j$ is an endpoint in
$\vartheta(\pi)$ if and only if $(j+1)$ is an endpoint in $\pi$ and
Claim $2$ follows, completing the proof of the theorem.
\end{proof}
As an illustration of the mapping $\vartheta\colon \mathcal{P}_k(n)
\longrightarrow\mathcal{B}_k(n-1)$ we give the following example
\begin{center}
\scalebox{0.5}[0.5]{\includegraphics*[50,710][620,860]{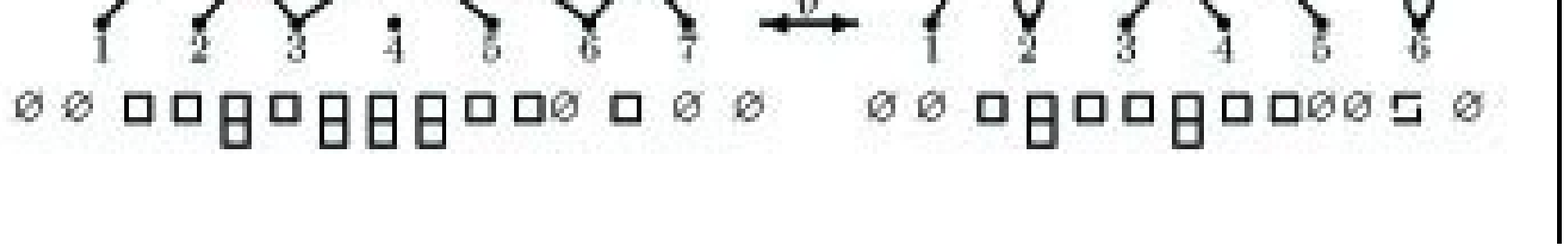}}
\end{center}
Theorem~\ref{T:k-noncross} implies by restriction a bijection
between $k$-noncrossing $2$-regular partitions and braids without
isolated points, denoted by $\mathcal{B}_k^\dagger(n)$. This is of
importance since the former cannot be enumerated via the reflection
principle while the latter can: braids without isolated points
``just'' lack the pair-step $(\varnothing,\varnothing)$ which
introduces a factor $e^x$ for the generating functions.
Consequently, we can enumerate $\mathcal{B}_k^\dagger$ using kernel
methods.

We will actually give an independent direct proof of this result.
For this purpose we interpret $k$-noncrossing braids without isolated
points as a subset of $k$-noncrossing partitions. For $\delta\in
\mathcal{B}^\dagger_k(n)$ we identify loops with isolated points and
crossing arcs $(i,j)$ and $(j,h)$, where $i<j<h$ by noncrossing
arcs. We accordingly arrive at the bijection
\begin{equation}\label{E:nobraid}
f\colon \mathcal{B}_k^\dagger(n)\longrightarrow
\{\pi\in \mathcal{P}_k(n)\mid \not \exists\,
(i_1,j_1),\dots ,(i_k,j_k);\quad i_1<\dots <i_k=j_1<\dots <j_k \,\} \ .
\end{equation}
Let $\mathcal{P}_{k,2}(n)$ denote the set of $2$-regular,
$k$-noncrossing partitions, i.e.~the set of $k$-noncrossing partitions without
arcs of the form $(i,i+1)$.
\begin{theorem}\label{T:bijection-2}
Let $k\in\mathbb{N}$, $k\ge 3$. Then we have the bijection
\begin{equation}\label{E:biject1}
\vartheta\colon \mathcal{P}_{k,2}(n)\longrightarrow
\mathcal{B}_k^\dagger(n-1) \ ,
\end{equation}
where $\vartheta$ is given by Theorem~\ref{T:k-noncross}.
\end{theorem}
\begin{proof}
By construction, $\vartheta$ maps tangled-diagrams over $[n]$
into tangled diagrams over $[n-1]$.
Since there exist no arcs of the form $(i,i+1)$, $\vartheta(\pi)$ is,
for any $\pi\in \mathcal{P}_{k,2}(n)$ loop-free. By construction,
$\vartheta$ preserves the orientation of arcs, whence $\vartheta(\pi)$
is a partition.\\
{\it Claim.} $\vartheta\colon \mathcal{P}_{k,2}(n)\longrightarrow
\mathcal{B}_k^\dagger(n-1)$ is well-defined.\\
We first prove that $\vartheta(\pi)$ is $k$-noncrossing.
Suppose there exist $k$ mutually crossing arcs, $(i_s,j_s)$, $s=1,\dots ,k$
in $\vartheta(\pi)$. Since $\vartheta(\pi)$ is a partition we have
$i_1<\dots <i_k<j_1<\dots <j_k$. Accordingly, we obtain for the partition
$\pi \in \mathcal{P}_{k,2}(n)$ the $k$ arcs
$(i_s,j_s+1)$, $s=1,\dots ,k$ where $i_1<\dots <i_k<j_1+1<\dots <j_k+1$,
which is impossible since $\pi$ is $k$-noncrossing.
We next show that $\vartheta(\pi)$ is a $k$-noncrossing braid. If
$\vartheta(\pi)$ is not a $k$-noncrossing braid, then according to
eq.~(\ref{E:nobraid}) $\vartheta(\pi)$
contains $k$ arcs of the form $(i_1,j_1),\dots (i_k,j_k)$ such that
$i_1< \dots <i_k=j_1<\dots <j_k$ holds. Then $\pi$ contains the arcs
$(i_1,j_1+1)$, $(i_k,j_k+1)$ where $i_1<\dots <i_k<j_1+1<\dots <j_k+1$,
which is impossible since these arcs are a set of $k$ mutually crossing arcs
and the claim follows.\\
{\it Claim.} $\vartheta$ is bijective.\\
Clearly $\vartheta$ is injective and it remains to prove surjectivity.
For any $k$-noncrossing braid $\delta$ there exists some $2$-regular
partition $\pi$ such that $\vartheta(\pi)=\delta$.
We have to show that $\pi$ is $k$-noncrossing.
Let $M'=\{(i_1,j_1),\dots ,(i_k,j_k)\}$ be a set of $k$ mutually
crossing arcs, i.e.~$i_1<\dots <i_k<j_1<\dots <j_k$. Then we have in
$\vartheta(\pi)$ the arcs $(i_s,j_s-1)$,
$s=1,\dots ,k$ and $i_1<\dots <i_k\le
j_1-1<\dots <j_k-1$. If $M=\{(i_1,j_1-1),\dots ,(i_k,j_k-1)\}$ is
$k$-noncrossing then we can conclude $i_k=j_1-1$. Therefore
$M=\{(i_1,j_1-1),\dots ,(i_k,j_k-1)\}$, where $i_k=j_1-1$ which is,
in view of eq.~(\ref{E:nobraid}) impossible in $k$-noncrossing braids.
By transposition we have thus proved that any $\vartheta$-preimage
is necessarily a $k$-noncrossing partition, whence the claim and the proof
of the theorem is complete.
\end{proof}

Theorem~\ref{T:bijection-2} allows for lattice path enumeration of
$\mathcal{P}_{k,2}(n)$. The main difficulty lies the kernel-computation
\cite{Mohanty79} and at present time there exists no such formula for
$k>3$. However, for $\mathcal{B}_3^\dagger(n-1)$ we have in the following
enumerative result.
\begin{theorem}\cite{MIRXIN,Reidys:073d}\label{T:braid}
The number of $3$-noncrossing braids without isolated points over $[n]$,
$\rho_3(n)$, is given by
\begin{eqnarray*}{\label{E:braidnum1}}
\rho_3(n) &=& \sum_{s\in\mathbb{Z}}\left[\beta_{n}(1,0,s)-\beta_{n}(1,-1,s)-
\beta_{n}(1,-4,s)+\beta_{n}(1,-3,s)\right.\\
& & -\beta_{n}(3,4,s)+\beta_{n}(3,3,s)+\beta_n(3,0,s)-\beta_{n}(3,1,s)\\
& & \left.
+\beta_{n}(2,5,s)-\beta_{n}(2,4,s)-\beta_{n}(2,1,s)+\beta_{n}(2,2,s))\right]
\ ,
\end{eqnarray*}
where $\beta_{n}(t,m,s)=\frac{t}{n+1}{n+1\choose s}{n+1\choose t+s}
{n+1\choose s+m}$. Furthermore $\rho_3(n)$ satisfies the recursion
\begin{equation}\label{E:recursion-b}
\alpha_1(n)\, \rho_{3}(n)+
\alpha_2(n)\, \rho_3(n+1)+
\alpha_3(n)\, \rho_{3}(n+2)-
\alpha_4(n)\, \rho_{3}(n+3)=0 \ ,
\end{equation}
where $\alpha_1(n)  =  8(n+2)(n+3)(n+1)$, $\alpha_2(n) =
3(n+2)(5n^2+47n+104)$, $\alpha_3(n) =  3(n+4)(2n+11)(n+7)$ and
$\alpha_4(n)= (n+9)(n+8)(n+7)$ and
\begin{equation}
\rho_{3}(n)\sim K \ 8^{n}n^{-7}(1+c_{1}/n+c_{2}/n^2+c_3/n^3),
\end{equation}
where $K=6686.408973$, $c_1=-28,\
c_2=455.77778$ and $c_3=-5651.160494$.
\end{theorem}
The theorem has two parts: the first is the exact formula resulting from the
kernel computation \cite{MIRXIN} and the second is the asymptotic formula
\cite{Reidys:073d}.
In \cite{MIRXIN} the exact formula is computed, the authors also prove an
asymptotic formula.
In \cite{Reidys:073d} an improved asymptotic formula is given which is based
the analytic theory of singular difference equations developed by Birkhoff
and Trjitzinsky \cite{Birkhoff,T:wimp}. To keep the paper self-contained we
prove Theorem~\ref{T:braid} in the Section~\ref{S:braid}.

\begin{remark}
The enumeration results for $\mathcal{B}_3^\dagger(n)$ summarized in
Theorem~\ref{T:braid} imply trivially the enumeration of $\mathcal{B}_3(n)$.
According to the duality between braids and partitions we have therefore
obtained the enumeration of $3$-noncrossing partitions.
\end{remark}


\section{Proof of Theorem~\ref{T:braid}}\label{S:braid}


We have $k=3$, i.e.~walks induced by the vacillating braid-tableaux
in $\mathbb{Z}^2$, starting and ending at $(1,0)$. Via the
reflection principle we reduce the enumeration of these walks which
remain in the first quadrant and never touch the diagonal $x=y$
to the enumeration of lattice walks in the first quadrant
starting and ending at $(1,0)$ and starting at $(1,0)$ and ending at $(0,1)$,
respectively. Let $h(i,j,l)$ be the number of walks of length $l$ that end
at $(i,j)$ and let $H(x,y;t)  =  \sum_{i,j,l}h(i,j,l)x^iy^jt^{l}$.
We set $\bar{x}=x^{-1}$.\\
{\bf Claim $1$.}
The series $H(x,0;t)$ and $H(0,\bar{x};t)$ satisfy
\begin{align}\label{E:11}
t^2x(x+1)H(x,0;t)&=PT_{x}(x^2Y_0-\bar{x}^2Y_{0}^3+\bar{x}^3Y_{0}^2)\\
\label{E:22}
t^2\bar{x}(\bar{x}+1)H(0,\bar{x};t)&=NT_{x}(x^2Y_0-\bar{x}^2Y_{0}^3+
\bar{x}^3Y_{0}^2),
\end{align}
where the operator $PT_{x}$($NT_{x}$) extracts positive(negative)
powers of $x$ in series of $\mathbb{Q}[x,\bar{x}][[t]]$.\\
To prove the Claim $1$ we observe that the kernel of
\begin{eqnarray*}
H(x,y;t)-x & = &
(x+y+\bar{x}+\bar{y}+x\bar{y}+y\bar{x}+y\bar{y}+x\bar{x})\,t^2 H(x,y;t) \\
&& -t^2\, (x\bar{y}+\bar{y})\, H(x,0;t)-t^2\, (y\bar{x}+\bar{x})\,H(0,y;t)
\end{eqnarray*}
is given by:
\begin{equation}{\label{E:braidkernel}}
K_{\mathcal{B}_3}(x,y;t)=xy-t^2(x^2y+xy^2+y+x+x^2+y^2+2xy) \ .
\end{equation}
$K_{\mathcal{B}_3}(x,y;t)$ is an irreducible polynomial of degree $2$
over $\mathbb{Q}(y,t)$ having the two roots $Y_0=Y_0(x,t)$ and
$Y_1=Y_1(x,t)$. Only $Y_0$ is a power series with positive coefficients
in $t^2$:
\begin{equation}{\label{E:Y_0}}
Y_0 =  \frac{1-t^2(x+2+\bar{x})-
\sqrt{(1-t^2(x+2+\bar{x}))^2-4t^4x\,(1+\bar{x})^2}}
{2t^2(\bar{x}+1)}
\end{equation}
i.e.~$Y_0(x,t)=(1+x)t^2+(x(x+1)(\bar{x}+1)^2)t^4+O(t^6)$.
Furthermore we have $Y_0\,Y_1=x$ and
\begin{equation}{\label{E:braidkernel2}}
x^2\,\bar{y} \,K_{\mathcal{B}_3}(\bar{x}y,y;t)= K_{\mathcal{B}_3}(x,y;t) =
x^3\,K_{\mathcal{B}_3}(\bar{x}y,\bar{x};t) \ .
\end{equation}
Eq.~(\ref{E:braidkernel2}) implies
$K_{\mathcal{B}_3}(\bar{x}Y_0,\bar{x};t)=
K_{\mathcal{B}_3}(\bar{x}Y_0,Y_0;t)=K_{\mathcal{B}_3}(x,Y_0;t)=0$ and
we accordingly obtain
\begin{align}
x^2Y_0&=t^2x(x+1)H(x,0;t)+t^2Y_0(Y_0+1)H(0,Y_0;t)\\
\bar{x}^2Y_0^3&=t^2\bar{x}Y_0(\bar{x}Y_0+1)H(\bar{x}Y_0,0;t)+
t^2Y_0(Y_0+1)H(0,Y_0;t)\\
\bar{x}^3Y_{0}^2&=t^2\bar{x}Y_{0}(\bar{x}Y_{0}+1)H(\bar{x}Y_{0},0;t)+
t^2\bar{x}(\bar{x}+1)H(0,\bar{x};t) \ .
\end{align}
We next eliminate the terms $H(0,Y_{0};t)$ and $H(\bar{x}Y_{0},0;t)$
and arrive at
\begin{equation}{\label{E:braideqn}}
x^{2}Y_{0}-\bar{x}^2Y_{0}^3+\bar{x}^3Y_{0}^2=
t^2x(x+1)H(x,0;t)+t^2\bar{x}(\bar{x}+1)H(0,\bar{x};t) \ .
\end{equation}
Since $t^2x(x+1)H(x,0;t)$ and $t^2\bar{x}H(0,\bar{x};t)$ have only
positive and negative powers of $x$, respectively, we can conclude
\begin{eqnarray*}
t^2x(x+1)H(x,0;t) &=& PT_{x}(x^2Y_0-\bar{x}^2Y_{0}^3+\bar{x}^3Y_{0}^2)\\
t^2\bar{x}(\bar{x}+1)H(0,\bar{x};t) &=& NT_{x}(x^2Y_0-
\bar{x}^2Y_{0}^3+\bar{x}^3Y_{0}^2) \ .
\end{eqnarray*}
{\bf Claim $2$.}
Let $CT_x$ denote the constant coefficient of a Laurent-series
$\sum_{i\in I}a_ix^i$. Then we have
\begin{equation}\label{E:claim1}
\rho_3(n)= [t^{2n+2}]CT_{x}((1-x-x^4+x^3)Y_{0}+
(-\bar{x}^{4}+\bar{x}^{3}+1-\bar{x})Y_{0}^{3}+(\bar{x}^{5}-
\bar{x}^{4}-\bar{x}+\bar{x}^2)Y_{0}^2)\ .
\end{equation}
To prove Claim $2$ we write $\rho_3(n) =[xt^{2n}]H(x,0;t)-[yt^{2n}]H(0,y;t)$
and interpret the terms $[xt^{2n}]H(x,0;t)$ and $[yt^{2n}]H(0,y;t)$
via eq.~(\ref{E:11}) and eq.~(\ref{E:22}):
\begin{align*}
[xt^{2n}]H(x,0;t) &=
[x^2t^{2n+2}]PT_{x}(x^2Y_{0}-\bar{x}^2Y_{0}^3+\bar{x}^3Y_{0}^2)
-[xt^{2n+2}]PT_{x}(x^2Y_{0}-\bar{x}^2Y_{0}^3+\bar{x}^3Y_{0}^2) \\
[yt^{2n}]H(0,y;t) &= [\bar{x}t^{2n}]H(0,\bar{x};t)\\
   &=[\bar{x}^2t^{2n+2}]\, NT_{x}(x^2Y_{0}-\bar{x}^2Y_{0}^3+\bar{x}^3Y_{0}^2)
-[\bar{x}t^{2n+2}]\, NT_{x}(x^2Y_{0}-\bar{x}^2Y_{0}^3+\bar{x}^3Y_{0}^2) \ .
\end{align*}
We can combine these equations and obtain
\begin{align*}
\rho_3(n) &= [t^{2n+2}]CT_{x}(\bar{x}^{2}-\bar{x}-x^2+x)(x^2Y_{0}-
              \bar{x}^2Y_{0}^{3}+\bar{x}^{3}Y_{0}^{2})\\
&{\label{E:12term}}=[t^{2n+2}]CT_{x}((1-x-x^4+x^3)Y_{0}+
(-\bar{x}^{4}+\bar{x}^{3}+1-\bar{x})Y_{0}^{3}+(\bar{x}^{5}-
\bar{x}^{4}-\bar{x}+\bar{x}^2)Y_{0}^2) \ .
\end{align*}
{\bf Claim $3$.} Suppose $Y_0$ is the solution of $K_{\mathcal{B}_3}(x,y;t)=0$
with positive coefficients in $t^2$ of eq.~(\ref{E:Y_0}). Then we have
\begin{equation}
[x^mt^{2n+2}]Y_{0}^{k}=\frac{k}{n+1}{n+1\choose
s}{n+1\choose k+s}{n+1\choose s+m} \ .
\end{equation}
Since $K_{\mathcal{B}_3}(x,Y_0;t)=0$, (eq.~(\ref{E:braidkernel})) we have
$Y_{0}=t^{2}(\bar{x}+1)(x+Y_{0})(1+Y_{0}).$ Let
$\mathcal{G}(t^2)=(\bar{x}+1)(x+t^2)(1+t^2)$. We derive
\begin{align*}
[t^{2n+2}]Y_{0}^{k}&=
\frac{k}{n+1}[t^{2(n+1-k)}](\bar{x}+1)^{n+1}(x+t^2)^{n+1}(1+t^2)^{n+1}\\
  &=\frac{k}{n+1}\left(\sum_{s=0}^{n+1-k}(\bar{x}+1)^{n+1}{n+1\choose s}
     {n+s\choose n+1-s-k}x^{n+1-s}\right).
\end{align*}
We can conclude from this
\begin{equation}
[x^{m}t^{2n+2}]Y_{0}^{k}=\frac{k}{n+1}\sum_{s=0}^{n+1}{n+1\choose
s}{n+1\choose k+s}{n+1\choose s+m}.
\end{equation}
and Claim $3$ follows.
In order to prove the first assertion of the theorem, we calculate the
first term $[t^{2n+2}]CT_{x}((1-x-x^4+x^3)Y_{0}$ of eq.~(\ref{E:claim1}).
The terms
$(-\bar{x}^{4}+\bar{x}^{3}+1-\bar{x})Y_{0}^{3}$ and $(\bar{x}^{5}-
\bar{x}^{4}-\bar{x}+\bar{x}^2)Y_{0}^2$
can be computed analogously:
\begin{align*}
[t^{2n+2}]CT_{x}((1-x-x^4+x^3)Y_{0}&=[x^{0}t^{2n+2}]Y_{0}-
[x^{-1}t^{2n+2}]Y_{0}-[x^{-4}t^{2n+2}]Y_{0}+[x^{-3}t^{2n+2}]Y_0\\
&=\sum_{s=0}^{n+1}(\beta_{n}(1,0,s)-\beta_{n}(1,-1,s)-
\beta_{n}(1,-4,s)+\beta_{n}(1,-3,s))\ ,
\end{align*}
where $\beta_{n}(k,m,s)=\frac{k}{n+1}{n+1\choose s}{n+1\choose
k+s}{n+1\choose s+m}$. Using eq.~(\ref{E:braidnum1}) the recursion
follows from Zeilberger's algorithms {\cite{Wilf}} using MAPLE.\\
{\bf Claim $4$.}
There exist some $K>0$ and $c_{1},c_{2},c_{3}\dots$ such that
\begin{equation}
\rho_{3}(n)\sim K \ 8^{n}n^{-7}(1+c_{1}/n+c_{2}/n^2+c_3/n^3\cdots).
\end{equation}
The theory of singular difference equations \cite{Birkhoff} guarantees
the existence of $3$ linearly independent formal series solutions (FSS) for
eq.~(\ref{E:recursion-b}). We set
\begin{equation}
\rho_{3}(n) = E(n)K(n) \quad E(n)=e^{\mu_0n\ln n+\mu_{1}n}n^{\theta}
\end{equation}
where $K(n)=\exp\{\alpha_{1}n^{\beta+\alpha_2n^{\beta-1/\rho+\cdots}}\}$,
$\alpha_1\neq 0$, $\beta=j/\rho$, and $0\leq j<\rho$.
We immediately derive setting $\lambda =e^{\mu_{0}+\mu_{1}}$
\begin{align*}
\frac{\rho_{3}(n+k)}{\rho_{3}(n)}&=n^{\mu_{0}k}\lambda^{k}
\{1+\frac{k\theta+k^2\mu_{0}/2}{n}+\cdots\} \\
  & \quad \ \exp\{\alpha_1\beta kn^{\beta-1}+\alpha_2(\beta-\frac{1}{\rho})
kn^{\beta-1/\rho-1+\cdots}\},
\end{align*}
and arrive at
\begin{align*}
0=1+&\frac{15}{8}\{1+\frac{\theta+\mu_0/2+\frac{27}{5}}{n}+
\cdots\}\xi\{1+(\alpha_1\beta
n^{\beta-1}+\alpha_2(\beta-1/\rho)n^{\beta-1/\rho-1}+\cdots)+\cdots\}\\
+&\frac{3}{4}\{1+\frac{2\theta+2\mu_0+\frac{21}{2}}{n}+\cdots\}
\xi^2\{1+(2\alpha_1\beta
n^{\beta-1}+2\alpha_2(\beta-1/\rho)n^{\beta-1/\rho-1}+\cdots)+\cdots\}\\
-&\frac{1}{8}\{1+\frac{3\theta+9\mu_0/2+18}{n}+\cdots\}
\xi^3\{1+(3\alpha_1\beta
n^{\beta-1}+3\alpha_2(\beta-1/\rho)n^{\beta-1/\rho-1}+\cdots)+\cdots\}.
\end{align*}
First we consider the maximum power of $n$, which is zero. In view
of $1=\frac{1}{8}n^{3\mu_0}\lambda^3$ we obtain $\mu_0=0$. This
implies $\rho=1$ since $\rho\geq 1$ and $\rho$ should be the
smallest integer s.t. $\rho\mu_0\in \mathbb{N}$. Equating the
constant terms again, we obtain that $\lambda$ is indeed a root of
the cubic polynomial $P(X)$
$
P(X) = 1+\frac{15}{8}X+\frac{3}{4}X^2-\frac{1}{8}X^3
$.
Therefore we have $\lambda=8$ or $-1$. Notice that $0\leq \beta <1$
implies $\beta=0$. Otherwise, equating the coefficient of
$n^{\beta-1}$ implies $\alpha_1=0$, which is impossible. It remains
to compute $\theta$. For this purpose we equate the coefficient of
$n^{-1}$, i.e.~
$
8\frac{15}{8}(\theta+\frac{27}{5})+8^2\frac{3}{4}
(\frac{21}{2}+2\theta)-8^3\frac{1}{8}(18+3\theta)=0
$
from which we can conclude $\theta=-7$. Since $\rho_3(n)$ is
monotone increasing $\rho_3(n)$ coincides with the only
monotonously increasing FSS, given by
\begin{equation}
\rho_{3}(n)\sim K \cdot 8^{n}\cdot
n^{-7}(1+c_{1}/n+c_{2}/n^2+c_3/n^3\cdots)
\end{equation}
for some $K>0$ and constants $c_1,c_2,c_3$ and the proof of the
Claim $4$ is complete. \\
Equating the coefficients of $n^{-2},\ n^{-3}$ and $n^{-4}$, ($2268+81c_1=0$,
$1683c_1+162c_2-26712=0$ and $-32547c_1+729c_2+129654+243c_3=0$)
we obtain $c_1=-28$, $c_2=455.778$ and $c_3=-5651.160494$ and finally we
compute $K=6686.408973$ numerically, completing the proof of
Theorem~\ref{T:braid}.

{\bf Acknowledgments.}
We are grateful to Prof.~W.Y.C.~Chen for stimulating discussions. This work
was supported by the 973 Project, the PCSIRT Project of the Ministry
of Education, the Ministry of Science and Technology, and the
National Science Foundation of China.

\bibliographystyle{amsplain}

\end{document}